\documentclass[graybox]{svmult}
\usepackage{mathptmx}       
\usepackage{helvet}         
\usepackage{courier}        
\usepackage{makeidx}         
\usepackage{graphicx}   
\usepackage{subcaption}
\usepackage{multicol}        
\usepackage[bottom]{footmisc}
\usepackage{amsmath,amssymb,amsfonts}        
\makeindex             

\usepackage{booktabs}

\newcommand{\ds}{\displaystyle}

\newcommand{\R}{\mathbb R}
\newcommand{\RtE}{RRE}
\newcommand{\RE}{\varphi_{\gamma}^{\texttt{Ryu}}}
\newcommand{\norm}[1]{{\left\|{#1}\right\|}}
\newcommand{\inner}[2]{{\left\langle {#1},{#2}\right\rangle}}
\newcommand{\tophi}{\xrightarrow[\varphi]{}}

\DeclareMathOperator{\prox}{prox}
\DeclareMathOperator*{\argmin}{argmin}
\DeclareMathOperator{\fix}{Fix}
\DeclareMathOperator{\dom}{dom}

\newtheorem{assumption}{Assumption}

\usepackage{enumitem}
\usepackage{mathtools}
\usepackage{cite}
\usepackage{hyperref} 
\usepackage{soul}
\hypersetup{
        colorlinks=true,   
        linkcolor=blue,    
        citecolor=blue,    
        urlcolor=cyan      
      }
\usepackage{accents}   
\begin{document}
\title*{A relaxed version of Ryu's three-operator splitting method for structured nonconvex optimization}
\titlerunning{Modified Ryu splitting for nonconvex optimization}
\author{Jan Harold Alcantara and Felipe Atenas}
\institute{Jan Harold Alcantara \at Center for Advanced Intelligence Project, RIKEN, 1-4-1 Nihombashi, Chuo City, Tokyo, 103-0027, Japan \email{janharold.alcantara@riken.jp}
\and Felipe Atenas \at School of Mathematics and Statistics, The University of Melbourne,  813 Swanston St, Parkville, Melbourne, VIC, 3052, Australia \email{felipe.atenas@unimelb.edu.au}}
%
%
\maketitle

\abstract{In this work, we propose a modification of Ryu's splitting algorithm for minimizing the sum of three functions, where two of them are convex with Lipschitz continuous gradients, and the third is an arbitrary proper closed function that is not necessarily convex. The modification is essential to facilitate the convergence analysis, particularly in establishing a sufficient descent property for an associated envelope function. This envelope, tailored to the proposed method, is an extension of the well-known Moreau envelope. Notably, the original Ryu splitting algorithm is recovered as a limiting case of our proposal. The results show that the descent property holds as long as the stepsizes remain sufficiently small. Leveraging this result, we prove global subsequential convergence to critical points of the nonconvex objective.
}

\section{Introduction}
\label{sec:1}

Modern optimization applications often exhibit a structured form that is well-suited to the application of decomposition techniques.  Operator splitting methods have emerged as powerful tools for efficiently solving complex optimization problems. Instances of operator splitting methods include the forward-backward splitting method \cite{Passty79} and the Douglas-Rachford splitting method \cite{DR,LM}. These algorithmic schemes have been pivotal in the development of techniques for image reconstruction, signal processing, and machine learning \cite{combettes2007douglas,cai2010split,combettes2011proximal,boyd2011distributed,glowinski2017splitting}.

The aforementioned methods were originally designed to solve maximal monotone problems with a two-block structure. In the context of optimization problems, it translates to the sum of two convex functions. Several generalizations to three-block problems have been proposed in the literature, such as the Davis-Yin operator splitting method \cite{davis2017three}, Ryu's three-operator splitting method \cite{ryu2020uniqueness}, and the Malitsky-Tam operator splitting method \cite{MT23} (which is also applicable to $n$-block problems). All these methods, in the context of optimization, have provable convergence guarantees for convex optimization problems. 


The literature is relatively scarce for methods that solve nonconvex optimization problems for the sum of three functions. Notable contributions in this direction include the Davis-Yin splitting to nonconvex problems analyzed in \cite{bian2021three}, the extension of Davis-Yin to four-operator splitting investigated in \cite{ALT24}, and an extension of Douglas-Rachford splitting examined in \cite{AT25} (that also works for $n$-operators). In this work, we propose a modification of Ryu's splitting method 
to solve nonconvex three-block optimization problems with a specific structure.

Consider the optimization problem \begin{equation} \label{primal-problem}
\min_{x\in \R^d}~\varphi(x) \coloneqq  f_1(x) + f_2(x) + f_3(x),
\end{equation} 
where $f_i: \R^d \to \R\cup \{ +\infty\}$ for $i=1,2,3$. 
Given $z_1^0,z_2^0 \in \R^d$, $\gamma >0$ and $\lambda >0$,  the (relaxed) \emph{Ryu's three-operator splitting method} \cite{ryu2020uniqueness} is given by the following iteration: for {$k\geq 0$},
\begin{equation} \label{Ryu:iteration}
    \left\{\begin{array}{rcl}
        x_1^k & = &\prox_{\gamma f_1}(z_1^k)  \\
        x_2^k & = & \prox_{\gamma f_2}(z_2^k + x_1^k) \\
        x_3^k & \in & \prox_{\gamma f_3}\big(x_1^k - z_1^k + x_2^k - z_2^k\big)\\
        \begin{pmatrix}
            z_1^{k+1} \\ z_2^{k+1}
        \end{pmatrix} & = & \begin{pmatrix}
            z_1^k \\ z_2^k
        \end{pmatrix} + \lambda \begin{pmatrix}
            x_3^k - x_1^k \\ x_3^k - x_2^k
        \end{pmatrix},
    \end{array}\right.
\end{equation}
 where $\prox_{\gamma f}$ denotes the \emph{proximal mapping} of $\gamma f$ defined in \eqref{e:proximal} below.

When the functions $f_1, f_2$ and $f_3$ in problem \eqref{primal-problem} are convex, \cite[Theorem 4]{ryu2020uniqueness} states the convergence of the method to minimizers of problem \eqref{primal-problem}. Our objective in this work is to investigate whether the convergence guarantees can be extended to a nonconvex setting.

In our convergence analysis, we adopt the ``envelope technique'', a strategy that has been employed in the analysis of several splitting algorithms for nonconvex problems \cite{themelis2018forward,themelis2020douglas,ALT24,AT25,Atenas25}. A key challenge, however, lies in the fact that the standard form of Ryu's splitting algorithm is not directly amenable to this type of analysis. {This is primarily because, as will become evident in the analysis, the default formulation induces an envelope that lacks sufficient proximal terms, which are essential for deriving stepsize intervals that ensure sufficient descent. }To address this, we consider the following relaxed variant of Ryu's method: given $z_1^0,z_2^0 \in \R^d$, $\alpha>0$, $\gamma >0$ and $\lambda >0$, we define the iterates for all {$k \geq 0$} as follows:
\begin{equation} \label{Ryu:iteration_relaxed}
    \left\{\begin{array}{rcl}
        x_1^k & = &\prox_{\gamma f_1}(z_1^k)  \\
        x_2^k & = & \prox_{\frac{\gamma}{\alpha} f_2}(\frac{z_2^k}{\alpha} + x_1^k) \\
        x_3^k & \in & \prox_{\gamma f_3}\big(x_1^k - z_1^k + x_2^k - z_2^k\big)\\
        \begin{pmatrix}
            z_1^{k+1} \\ z_2^{k+1}
        \end{pmatrix} & = & \begin{pmatrix}
            z_1^k \\ z_2^k
        \end{pmatrix} + \lambda \begin{pmatrix}
            x_3^k - x_1^k \\ x_3^k - x_2^k
        \end{pmatrix}.
    \end{array}\right.
\end{equation}
In particular, we modify the $f_2$-proximal step in \eqref{Ryu:iteration} to obtain \eqref{Ryu:iteration_relaxed}. When $\alpha=1$, this reduces to Ryu's splitting algorithm.

Throughout this paper, we adopt the following blanket assumptions.

\begin{assumption}[Blanket assumption] \label{a:blanket}
    Suppose the following conditions are satisfied.
    \begin{enumerate}[label=(\alph*)]
        \item For $i=1,2$, the function $f_i: \R^d \to \R$  is a convex $L_i$-smooth function, that is, $\nabla f_i$ is globally Lipschitz continuous with modulus $L_i >0$.
        \item The function $f_3: \R^d \to \R \cup \{+\infty\}$ is proper and lower semicontinuous (lsc for short).
        \item Problem \eqref{primal-problem} has a nonempty set of solutions.
    \end{enumerate}
\end{assumption}

\begin{remark}[On the simplified nonconvex setting]
The convexity in Assumption~\ref{a:blanket}(a) can be relaxed. However, we impose it in our setting to simplify the analysis. Similar results can be obtained when $f_1$ and $f_2$ are merely $L_1$- and $L_2$-smooth, respectively, by employing a strategy similar to that in \cite{themelis2018forward,bian2021three}.
\end{remark}

This paper is organized as follows. In Section~\ref{s:prelim}, we introduce the notation and some variational analysis results we use throughout this paper. In Section~\ref{s:envelope}, we define the merit function we use as the foundation of our analysis $-$ the Ryu's three-operator splitting envelope $-$ and establish key properties relevant to the convergence analysis of the proposed method. Section~\ref{s:convergence} is dedicated to investigating the convergence properties of the method \eqref{Ryu:iteration_relaxed}. In particular, we show that the defined envelope satisfies a sufficient decrease condition, and then we exploit this property to prove that all cluster points of the generated sequence are critical points of the objective function of problem \eqref{primal-problem}. { In Section~\ref{s:numerics}, we illustrate the applicability of the method to solve a nonconvex formulation of the matrix completion/factorization problem.} Finally, in Section~\ref{s:conclusion} we comment on some ongoing works and future research directions.

\section{Preliminaries and notation} \label{s:prelim}

Throughout this paper, $\langle \cdot, \cdot \rangle$ denotes an inner product  in $\R^d$, and $\| \cdot \|$ its induced norm.  We shall make use of the following technical result.

\begin{lemma} \label{lemma:identity} For all $a,b,c,d \in \R^d$, it holds
    $$\|a-b\|^2 - \|a - c\|^2 = - \|b - c\|^2 + 2 \langle  c - b, a - b \rangle.$$
\end{lemma}

A function $\varphi: \R^d \to \R \cup \{+\infty\}$ is called proper when its domain, the set  $\dom(\varphi) = \{ x \in \R^d: \varphi(x) < +\infty\}$, is nonempty. We say $\varphi$ is lsc if at any $x \in \dom(\varphi)$, $f(x) \leq \liminf_{y\to x} f(y)$, and it is convex if for all $x,y \in \dom(\varphi)$, and all $\beta \in (0,1)$, $\varphi(\beta x + (1-\beta) y) \leq \beta \varphi(x) + (1-\beta)\varphi(y)$. A function $\Phi: \R^d \to \R$ is said to be $L$-smooth if it is differentiable and its gradient is $L$-Lipschitz continuous, that is, for all $x,y \in \R^d$, $\norm{\nabla \Phi (x) - \nabla \Phi (y)} \leq L \|x-y\|$. In the next result, we recall some propeties of $L$-smooth functions that will be important in our analysis (see, for instance, \cite[Theorem 5.8]{Beck17}). 

\begin{lemma} \label{descent-lemma}
    Suppose $f: \R^d \to \R$ is $L$-smooth. Then, \begin{equation*}
        (\forall x,y \in \R^d) \; |f(x) - f(y) - \langle \nabla f(y), x -y \rangle| \leq \dfrac{L}{2}\|y-x\|^2.
    \end{equation*} If, in addition, $f$ is convex, then \begin{equation*}
        \langle \nabla f(x) - \nabla f(y), x-y \rangle \geq \dfrac{1}{2L}\|\nabla f(x) - \nabla f(y)\|^2.
    \end{equation*}
\end{lemma}


For a proper function $\varphi: \R^d \to \R \cup \{+\infty\}$, and $x \in \dom(\varphi)$, we denote by  $\hat{\partial} \varphi(x)$ the Fr\'echet (or regular)
subdifferential of $\varphi$ at $x$, defined as \begin{equation*} \hat{\partial}
	\varphi(x) = \left\{ v \in \R^d : \liminf_{y \to x}
	\frac{\varphi(y)-\varphi(x) - \langle v, y - x \rangle}{\|
		y - x\|} \ge 0\right\}.  \end{equation*}   The limiting (or general) subdifferential of $\varphi$ at $x$, denoted $\partial \varphi (x)$, is defined as     \begin{equation*} \partial \varphi(x) = \limsup_{y \tophi x} \hat{\partial}\varphi(x),
\end{equation*} where $y \tophi x$ denotes convergence in the attentive sense, that is, $y \to
x $ and $\varphi(y) \to \varphi(x)$. When $\varphi$ is smooth, then $\hat{\partial}\varphi(x) = {\partial}\varphi(x) = \{\nabla \varphi (x)\}$. If $\varphi$ is proper lsc convex, then the Fr\'echet and limiting subdifferentials coincide with the subdifferential of convex analysis, namely, $$ \{ v \in \R^d: (\forall y \in \R^d) \; \varphi(y) \ge \varphi(x) + \langle v, y - x \rangle\}.$$  

A set of points of particular interest is the zeros of the subdifferential operator. We say $\bar{x} \in \R^d$ is a critical point of $\varphi$ if $0\in \partial \varphi(\bar{x})$. If $\varphi$ is convex, critical points are exactly the global minimizers of $\varphi$.

\begin{remark}[Critical points of the sum] \label{r:critical}
    Under Assumption \ref{a:blanket} and for $\varphi$ defined in \eqref{primal-problem}, in view of \cite[Exercise 8.8(c)]{Rockafellar_Wets_2009}, $\bar{x}$ is a critical point of $\varphi$ if
\begin{equation*}
    0 \in  \nabla f_1(\bar{x}) + \nabla f_2(\bar{x}) + \partial f_3 (\bar{x}).
\end{equation*}

\end{remark}


Given a function $\varphi: \R^d \to \R \cup \{+\infty\}$, a parameter $\gamma >0$, and a point $z \in \R^d$, the proximal operator of $\gamma \varphi$ at $x$ is defined as \begin{equation} \label{e:proximal}
    \prox_{\gamma \varphi} (z) \coloneqq \argmin\limits_{y\in\R^d} ~\varphi(y) + \frac{1}{2\gamma}\norm{y-z}^2.
\end{equation} The associated optimal value function is known as the Moreau envelope, defined as follows:
\begin{equation*}
    {\varphi}_{\gamma }^{\texttt{Moreau}} (z) \coloneqq \min_{y\in\R^d} ~\varphi(y) + \frac{1}{2\gamma}\norm{y-z}^2.
\end{equation*}

We say $\varphi$ is prox-bounded if, for some $\gamma >0$, $\varphi(\cdot) + \frac{1}{2\gamma}\|\cdot\|^2$ is bounded from below. The supremum $\gamma_{\varphi}>0$ of such parameters $\gamma$ is called the threshold of prox-boundedness. If $\varphi$ is a proper lsc prox-bounded function with threshold $\gamma_{\varphi} > 0$, then for any $\gamma \in (0, \gamma_{\varphi})$,   $\prox_{\gamma \varphi}$ is nonempty and compact-valued, and $\varphi_{\gamma \varphi}^{\texttt{Moreau}}$ is finite-valued \cite[Theorem 1.25]{Rockafellar_Wets_2009}. In particular, if $\varphi$ is a proper lsc convex function, then $\gamma_{\varphi}=+\infty$,  and $\prox_{\gamma \varphi}$ is a single-valued mapping \cite[Theorem 12.12, Theorem 12.17]{Rockafellar_Wets_2009}.

\begin{remark}[On the well-definedness of proximal operations]
\label{remark:welldefinedprox}
Under Assumption \ref{a:blanket}, 
    $\prox_{\gamma f_1}$ and $\prox_{\gamma f_2}$ are single-valued since $f_1$ and $f_2$ are proper lsc convex, while $\prox_{\gamma f_3}$ is well-defined for $\gamma < \frac{1}{L_1+L_2}$, as $f_3$ is prox-bounded with threshold at least $\frac{1}{L_1+L_2}$ from \cite[Remark 3.1]{themelis2020douglas}.

\end{remark}



\section{An envelope for Ryu's splitting method} \label{s:envelope}

Our convergence analysis for \eqref{Ryu:iteration_relaxed} under Assumption \ref{a:blanket} builds on the approach in \cite{ALT24,AT25,themelis2020douglas}, which utilizes an envelope function, akin to the Moreau envelope, well-suited to the corresponding iterative method. We begin our analysis by motivating such merit function.

\begin{proposition} \label{lemma:Ryu:envelope-derivation}
    Suppose that Assumption \ref{a:blanket} holds, and let $\alpha>0$ and $ (z_1^k, z_2^k) \in \R^d \times \R^d$. Then, $x_3^k$ defined in \eqref{Ryu:iteration_relaxed} solves the following minimization problem 
    \begin{equation*}
        \min_{y \in \R^d} \left\{f_3(y) + \displaystyle\sum_{i=1}^2 \left[ f_i(x_i^k) + \langle y - x_i^k,  \nabla f_i(x_i^k)   \rangle  + \dfrac{1}{2\gamma_i}\|y - x_i^k\|^2  \right]
    \right\},
    \end{equation*}
    where $\gamma_1 \coloneqq \frac{\gamma}{\alpha}$ and $\gamma_2 \coloneqq \frac{\gamma}{1-\alpha}$.\footnote{We adopt the convention that $\frac{c}{0}=\infty$ and $\frac{d}{\infty}=0$ for any $c> 0$ and $d\in \R$. } 
\end{proposition}

\begin{proof}

From the $x_3$-update in \eqref{Ryu:iteration_relaxed}, we have
\begin{equation}
    x_3^k \in \argmin_{y \in \R^d}\left\{f_3(y) + \dfrac{1}{2\gamma}\|y - (x_1^k - z_1^k + x_2^k - z_2^k) \|^2
    \right\}.
    \label{eq:x3^k}
\end{equation} 
Meanwhile, the first-order optimality conditions of the $x_1$-update and the $x_2$-update yield, respectively, \begin{align}
    z_1^k &= \gamma \nabla f_1(x_1^k) + x_1^k{\text{, and}} \label{eq:z1^k} \\
    z_2^k &= \gamma \nabla f_2(x_2^k) + \alpha x_2^k  -\alpha  x_1^k .\label{eq:z2^k}
\end{align}  Hence, \begin{align*}
    x_1^k - z_1^k + x_2^k - z_2^k& = x_1^k - \left( \gamma \nabla f_1(x_1^k) + x_1^k  \right) + x_2^k - \left(  \gamma \nabla f_2(x_2^k) + \alpha x_2^k  -\alpha  x_1^k\right)\\
    & = \alpha x_1^k - \gamma \nabla f_1(x_1^k) + (1-\alpha)x_2^k - \gamma \nabla f_2(x_2^k) \\
    & = \alpha \left( x_1^k - \frac{\gamma}{\alpha}\nabla f_1(x_1^k)\right) + (1-\alpha) \left( x_2^k - \frac{\gamma}{1-\alpha}\nabla f_2(x_2^k)\right),  
\end{align*} which, combined with \eqref{eq:x3^k}, yields \begin{equation}
    x_3^k \in \argmin_{y \in \R^d}\left\{f_3(y) + \dfrac{1}{2\gamma}\left\|y - \left(\sum_{i=1}^2\alpha_i(x_i^k - \gamma_i \nabla f_i(x_i^k) \right) \right\|^2
    \right\} 
    \label{eq:x3-simplified}
\end{equation} for $\alpha_1 = \alpha$ and $\alpha_2 = 1 - \alpha$. Following the calculations in \cite[Theorem 5.5]{AT25}, by expanding the squared norm, dropping the constant term $\left\|\sum_{i=1}^2\alpha_i(x_i^k - \gamma_i \nabla f_i(x_i^k) \right\|^2$, and  adding the constant term $\sum_{i=1}^2 f_i(x_i^k)$, we get the desired result. \qed
\end{proof}

In view of Proposition~\ref{lemma:Ryu:envelope-derivation}, we define the \emph{Relaxed Ryu envelope} (\RtE), for all $(z_1,z_2) \in \R^d \times \R^d$, by
\begin{equation} \label{eq:RE}
\varphi_{\gamma}^{\texttt{Ryu}}(z_1,z_2) \coloneqq  \min_{y \in \R^d} \left\{f_3(y) + \displaystyle\sum_{i=1}^2 \left[ f_i(x_i) + \langle y - x_i,  \nabla f_i(x_i)   \rangle  + \dfrac{1}{2\gamma_i}\|y - x_i\|^2  \right]
    \right\},
    \end{equation}
which will serve as our merit function. Meanwhile, the corresponding set-valued iteration operator associated with \eqref{Ryu:iteration_relaxed} is given by \begin{equation*}
    T_{\gamma}^{\texttt{Ryu}}: (z_1, z_2) \mapsto (z_1 + \lambda( x_3 - x_1), z_2 + \lambda (x_3 - x_2)),
\end{equation*} where \begin{equation} \label{eq:T-x}
    \left\{\begin{array}{rcl}
         x_1 & = &\prox_{\gamma f_1}(z_1)  \\
        x_2 & = & \prox_{\frac{\gamma}{\alpha} f_2}(\frac{z_2}{\alpha} + x_1) \\
        x_3 & \in & \prox_{\gamma f_3}\big(x_1 - z_1 + x_2 - z_2\big).
    \end{array}\right.
\end{equation}
In view of Remark \ref{remark:welldefinedprox}, note that $T_{\gamma}^{\texttt{Ryu}}$ is nonempty- and compact-valued for any $(z_1,z_2) \in \R^d \times \R^d$ provided that $\gamma<\frac{1}{L_1+L_2}$.

{
\begin{remark}\label{rem:limitingcase}
    When $\alpha=1$, we get $\gamma_2 = \infty$ and so the term $\dfrac{1}{2\gamma_i}\|y - x_i^k\|^2 $ in \eqref{eq:RE} vanishes when $i=2$. The absence of this proximal term makes the envelope-based analysis inapplicable. As will be evident in the subsequent development (see Propositions~\ref{Ryu:sandwich} and \ref{p:fix-crit}, as well as the main results in Theorems~\ref{thm:Ryu_sufficient_descent} and \ref{t:subseq-conv}), this proximal term plays a crucial role in ensuring the existence of a stepsize that guarantees sufficient descent of the proposed merit function.
\end{remark}
}

\begin{remark}[Relationship between the Moreau envelope and the \RtE] \label{r:Moreau-R3E}
    Observe that from \eqref{eq:x3-simplified}, we have \begin{multline*}
        \RE(z_1,z_2) = \varphi_{\gamma f_3}^{\texttt{Moreau}}\left(\sum_{i=1}^2\alpha_i(x_i - \gamma_i \nabla f_i(x_i) )\right) \\- \left\|\sum_{i=1}^2\alpha_i(x_i - \gamma_i \nabla f_i(x_i) )\right\|^2 + \sum_{i=1}^2 f_i(x_i). 
    \end{multline*}
\end{remark}

We now establish some properties of the envelope function.
The following result states that the \RtE~inherits continuity properties of the Moreau envelope (cf. \cite[Proposition 4.2]{themelis2018forward} and \cite[Proposition 3.2]{themelis2020douglas}).

\begin{proposition}[Continuity of \RtE] \label{p:loc-lip}
The \RtE~is a real-valued and locally Lipschitz continuous function.   
\end{proposition}

\begin{proof}
     Since $\prox_{\gamma f_1}$  and $\prox_{\frac{\gamma}{\alpha} f_2}$ are nonexpansive \cite[Theorem 6.42(b)]{Beck17}, then the maps $z_1 \mapsto x_1$ and $(z_1,z_2) \mapsto x_2$ defined in \eqref{eq:T-x} are (globally) Lipschitz continuous. Furthermore, from Assumption~\ref{a:blanket}, $\nabla f_i$ is (globally) Lipschitz continuous, for $i=1,2$. Then, as the Moreau envelope is locally Lipschitz continuous \cite[Example 10.32]{Rockafellar_Wets_2009}, the conclusion follows from Remark~\ref{r:Moreau-R3E}.\qed
\end{proof}

Next, we show some sandwich-type bounds relating the \RtE~and the original objective function in problem \eqref{primal-problem} (cf. \cite[Proposition 4.3]{themelis2018forward} and \cite[Proposition 3.3]{themelis2020douglas}).

\begin{proposition} \label{Ryu:sandwich}
Suppose Assumption~\ref{a:blanket} holds,  $\gamma \in (0, \frac{1}{L_1+L_2})$,  and $\lambda,\alpha>0$. Then, given $(z_1, z_2) \in \R^d \times \R^d$, consider $(x_1,x_2,x_3)\in\R^d \times \R^d \times \R^d$ defined by \eqref{eq:T-x}. Then, 
\begin{itemize}
    \item[(i)] For $\gamma_1 = \frac{\gamma}{\alpha}$ and $\gamma_2 = \frac{\gamma}{1-\alpha}$,\begin{multline*}
\varphi_{\gamma}^{\texttt{Ryu}}(z_1, z_2) \leq \min\left\{\varphi(x_1) +  \frac{1}{2}\left(L_2+\frac{1}{\gamma_2}\right)\|x_1 - x_2\|^2,\right.\\\left.\varphi(x_2) +  \frac{1}{2}\left(L_1+\frac{1}{\gamma_1}\right)\|x_1 - x_2\|^2\right\}.
    \end{multline*}
    \item[(ii)] $\varphi_{\gamma}^{\texttt{Ryu}}(z_1, z_2)  \geq  \varphi(x_3) + \frac{\alpha - \gamma L_1}{2\gamma} \norm{x_3-x_1}^2+ \frac{(1-\alpha) - \gamma L_2}{2\gamma}\norm{x_3-x_2}^2$. 
    \item[(iii)] If $\alpha\in (0,1)$ and $\gamma \leq \min \left\{ \frac{\alpha}{L_1}, \frac{1-\alpha}{L_2}\right\}$, then $\varphi_{\gamma}^{\texttt{Ryu}}(z_1, z_2)  \geq  \varphi(x_3) $.

\end{itemize}

\end{proposition}

\begin{proof}
    Take $y = x_1$ in \eqref{eq:RE} and apply Lemma~\ref{descent-lemma} to $f = f_2$  to obtain\begin{equation*}
    \begin{array}{rcl}
         \RE(z_1, z_2) &\leq& f_3(x_1)+f_1(x_1) + f_2(x_2)+ \langle \nabla f_2(x_2), x_1 - x_2\rangle  + \frac{1}{2\gamma_2}\|x_1 - x_2\|^2  \\
         & \leq & \varphi(x_1) + \dfrac{1}{2}\left(L_2+\frac{1}{\gamma_2}\right)\|x_1 - x_2\|^2.
    \end{array}
\end{equation*} Similarly, taking $y = x_2$ in  \eqref{eq:RE} and applying Lemma~\ref{descent-lemma} to $f = f_1$ yields \begin{equation*}
    \begin{array}{rcl}
         \RE(z_1, z_2) 
         & \leq & \varphi(x_2) + \dfrac{1}{2}\left(L_1+\frac{1}{\gamma_1}\right)\|x_2 - x_1\|^2.
    \end{array}
\end{equation*} From these, we immediately get (i).  Moreover,  since $y = x_3$ minimizes the problem in \eqref{eq:RE},
\begin{equation*}
    \begin{array}{rcl}
         \RE(z_1,z_2) & = & f_3(x_3) + \displaystyle\sum_{i=1}^2 \left[ f_i(x_i) + \langle \nabla f_i(x_i), x_3 - x_i\rangle +  \dfrac{1}{2\gamma_i}\|x_3-x_i\|^2  \right] \\
         & \geq & f_3(x_3) + \displaystyle\sum_{i=1}^2 \left[ f_i(x_3) -  \dfrac{L_i}{2}\|x_3-x_i\|^2 +  \dfrac{1}{2\gamma_i}\|x_3-x_i\|^2 \right]\\
         & = &   \varphi(x_3) + \frac{\alpha - \gamma L_1}{2\gamma} \norm{x_3-x_1}^2+ \frac{(1-\alpha) - \gamma L_2}{2\gamma}\norm{x_3-x_2}^2,
    \end{array}
\end{equation*} where the inequality holds by Lemma \ref{descent-lemma} and the last equality holds by plugging in $\gamma_1=\frac{\gamma}{\alpha}$ and $\gamma_2 = \frac{\gamma}{1-\alpha}$. This completes the proof of (ii). Finally, part (iii) immediately follows from (ii). \qed 
\end{proof}

Note that the iteration in \eqref{Ryu:iteration_relaxed} is designed to find a fixed point of the relaxed Ryu splitting operator \( T_{\gamma}^{\texttt{Ryu}} \), that is, a point in the set
\[
\fix T_{\gamma}^{\texttt{Ryu}} \coloneqq  \left\{ (z_1, z_2) \in \mathbb{R}^d \times \mathbb{R}^d : (z_1, z_2) \in T_{\gamma}^{\texttt{Ryu}}(z_1, z_2) \right\}.
\]
In the next proposition, we establish a connection between such fixed points and the notion of criticality, as commonly used in optimization.

\begin{proposition} \label{p:fix-crit}
Suppose Assumption~\ref{a:blanket} holds, and let $\gamma \in (0, \frac{1}{L_1+L_2})$ and $\alpha,\lambda>0$.   Then, $(\bar{z}_1, \bar{z}_2) \in \fix T_{\gamma}^{\texttt{Ryu}}$ {if and only if} $\bar{x} := \bar{x}_1= \bar{x}_2= \bar{x}_3$, where
\begin{equation} \label{eq:T-x-stationary}
    \left\{\begin{array}{rcl}
         \bar{x}_1 & = &\prox_{\gamma f_1}(\bar{z}_1)  \\
        \bar{x}_2 & = & \prox_{\frac{\gamma}{\alpha} f_2}(\frac{\bar{z}_2}{\alpha} + \bar{x}_1) \\
        \bar{x}_3 & \in & \prox_{\gamma f_3}\big(\bar{x}_1 - \bar{z}_1 + \bar{x}_2 - \bar{z}_2\big).
    \end{array}\right.
\end{equation} Furthermore, such $\bar{x}$ is a critical point of \eqref{primal-problem}, and 
$\varphi (\bar x) = \varphi_{\gamma}^{\texttt{Ryu}}(\bar{z}_1, \bar{z}_2)$. In particular, { if $\alpha\in (0,1)$ and $\gamma \leq \min \left\{ \frac{\alpha}{L_1}, \frac{1-\alpha}{L_2}\right\}$} 
, then \begin{equation*}
    \min \varphi = \min \RE.
\end{equation*}
\end{proposition}

\begin{proof}
    It is straightforward from the definition of $T_{\gamma}^{\texttt{Ryu}}$ that  $(\bar{z}_1, \bar{z}_2) \in \fix T_{\gamma}^{\texttt{Ryu}}$ is equivalent to having $\bar{x} := \bar{x}_1= \bar{x}_2= \bar{x}_3$ satisfying the conditions in \eqref{eq:T-x-stationary}. Hence, it suffices to prove that such $\bar{x}$ is, in this case, always a critical point of $\varphi$. Evaluating the first-order optimality conditions of the $x_3$-step, namely, \begin{equation} \label{Ryu:OC3}
        0 \in \gamma \partial f_3(x_3) + x_3 - x_1 + z_1 - x_2 + z_2
    \end{equation} at $(x_1,x_2,x_3)= (\bar{x},\bar{x},\bar{x})$ and $z_i = \bar{z}_i$, for $i=1,2$, yields \begin{equation*} \label{Ryu:OC3-xbar}
    0 \in \gamma\partial f_3(\bar{x})  + \bar{z}_1 - \bar{x} + \bar{z}_2.
\end{equation*} Adding this inclusion with \eqref{eq:z1^k} and \eqref{eq:z2^k} using the substitutions $x_1^k \leftarrow \bar{x}$ and  $x_2^k \leftarrow \bar{x}$, we obtain \begin{equation*}
    0 \in \gamma \big(\nabla f_1(\bar{x}) + \nabla f_2(\bar{x}) + \partial f_3(\bar{x})\big) ,
\end{equation*} that is, $0 \in \partial \varphi(\bar{x})$ by Remark~\ref{r:critical}. 


Furthermore, given  $(\bar{z}_1, \bar{z}_2) \in \fix T_{\gamma}^{\texttt{Ryu}}$, Proposition~\ref{Ryu:sandwich}(i)\&(ii)  
imply that $\varphi(\bar{x}) = \RE(\bar{z}_1,\bar{z}_2)$. 
Finally, in view of Proposition~\ref{Ryu:sandwich}(iii), 
 for any $(z_1,z_2) \in \R^d \times \R^d$, and $\bar{x} \in \argmin \varphi$, \begin{equation*}
    \RE(\bar{z}_1,\bar{z}_2) = \varphi(\bar{x}) \leq \varphi(x_3) \leq \RE(z_1,z_2).
\end{equation*} Hence, $(\bar{z}_1,\bar{z}_2) \in \argmin \RE$ and $\min \varphi = \min \RE$.\qed
\end{proof}

At this point, we have built the necessary tools regarding the \RtE~for the analysis of convergence of the proposed relaxed Ryu splitting method. In the next section, we show subsequential convergence of the method under standard assumptions in the literature.

\section{Convergence of modified Ryu's three-operator splitting method via envelopes} \label{s:convergence}
To establish the (subsequential) convergence of the iterative method in \eqref{Ryu:iteration_relaxed}, we follow the approach introduced in~\cite{themelis2020douglas} for the Douglas--Rachford splitting method. Specifically, the core argument relies on a sufficient decrease property satisfied by the \RtE.

\subsection{Sufficient decrease property for \RtE}
We first prove three technical lemmas that we will use in the main result of this section. We use the following notation: for $i=1,2$, \begin{equation}\label{Ryu:notation}
    \begin{array}{rcl}
        \Delta x_i^k &=& x_i^{k+1} - x_i^k, \\
         \Delta g_i^k &=& \nabla f_i(x_i^{k+1}) - \nabla f_i(x_i^k) , \\
         \Delta z_i^k & = & z_i^{k+1} - z_i^k.
    \end{array}
\end{equation}

\begin{lemma} \label{lemma:diff_of_sqrd_norms}
Under Assumption~\ref{a:blanket}, the sequences $(x_i^k)_{k}$ for $i=1,2,3$, and $(z_i^k)_k$ for $i=1,2$ generated by \eqref{Ryu:iteration_relaxed}
 satisfy: 
\begin{equation} \label{Ryu_relaxed:x_3-x_1}
    \|x_3^k - x_1^k\|^2 - \|x_3^k - x_1^{k+1}\|^2 = \left( \dfrac{2}{\lambda}- 1 \right)\|\Delta x_1^k\|^2 + \dfrac{2\gamma}{\lambda} \langle \Delta x_1^k,   \Delta g_1^k \rangle
\end{equation} and \begin{equation} \label{Ryu_relaxed:x_3-x_2}
    \|x_3^k - x_2^k\|^2 - \|x_3^k - x_2^{k+1}\|^2 = \left(\dfrac{2\alpha}{\lambda}-1 \right) \|\Delta x_2^k\|^2 + \dfrac{2\gamma}{\lambda} \langle \Delta x_2^k,   \Delta g_2^k \rangle  - \dfrac{2\alpha}{\lambda} \langle \Delta x_2^k,   \Delta x_1^k\rangle.
\end{equation}
\end{lemma}

\begin{proof}
Using the notations \eqref{Ryu:notation}, we have from   \eqref{eq:z1^k} and \eqref{eq:z2^k} that
\begin{align} 
    \Delta z_1^k & = \gamma \Delta g_1^k + \Delta x_1^k {\text{, and}}  \label{eq:Deltaz1}\\
    \Delta z_2^k & = \gamma \Delta g_2^k + \alpha \Delta x_2^k - \alpha \Delta x_1^k. \label{eq:Deltaz2}
\end{align}
On the other hand, Lemma~\ref{lemma:identity} and the $(z_1,z_2)$-update in \eqref{Ryu:iteration_relaxed} yield 
\begin{equation*}
     \|x_3^k - x_i^k\|^2 - \|x_3^k - x_i^{k+1}\|^2 = - \|\Delta x_i^k\|^2 + 2 \langle \Delta x_i^k,  x_3^k - x_i^k\rangle = - \|\Delta x_i^k\|^2 + \dfrac{2}{\lambda} \langle \Delta x_i^k,  \Delta z_i^k\rangle  .
\end{equation*} 
Together with \eqref{eq:Deltaz1} and \eqref{eq:Deltaz2}, we get \eqref{Ryu_relaxed:x_3-x_1} and \eqref{Ryu_relaxed:x_3-x_2}, respectively. \qed 
\end{proof}

\begin{lemma} \label{lemma:Ryu-lemma:L-smooth} Under Assumption~\ref{a:blanket}, the sequences $(x_i^k)_{k}$ for $i=1,2,3$, and $(z_i^k)_k$ for $i=1,2$ generated by \eqref{Ryu:iteration}
 satisfy: \begin{multline*}
    \sum_{i=1}^2 f_i(x_i^k) - f_i(x_i^{k+1}) - \langle \nabla f_i(x_i^{k+1}) , x_3^k - x_i^{k+1} \rangle + \langle \nabla f_i(x_i^k) , x_3^k - x_i^k \rangle \\ \geq \sum_{i=1}^2\left(\dfrac{1}{2L_i}- \dfrac{\gamma}{\lambda} \right)\| \Delta g_i^k\|^2  - \dfrac{1}{\lambda}\langle \Delta g_1^k,   \Delta x_1^k \rangle  \\   - \dfrac{\alpha }{\lambda}\langle \Delta g_2^k,    \Delta x_2^k \rangle + \dfrac{\alpha}{\lambda}\langle \Delta g_2^k,   \Delta x_1^k \rangle .
\end{multline*}    
\end{lemma}

\begin{proof} For $i=1,2$, we have from Lemma~\ref{descent-lemma} and the $(z_1,z_2)$-update rule in \eqref{Ryu:iteration_relaxed} that
\begin{align}
    &f_i(x_i^k) - f_i(x_i^{k+1}) - \langle \nabla f_i(x_i^{k+1}) , x_3^k - x_i^{k+1} \rangle + \langle \nabla f_i(x_i^k) , x_3^k - x_i^k \rangle \notag \\ 
         = & f_i(x_i^k) - f_i(x_i^{k+1}) - \langle \nabla f_i(x_i^{k+1}) , x_i^k - x_i^{k+1} \rangle - \langle  \Delta g_i^k , x_3^k - x_i^k \rangle \notag \\
         \geq & \dfrac{1}{2L_i}\| \Delta g_i^k\|^2 - \langle \Delta g_i^k, x_3^k - x_i^k \rangle \notag\\
         = & \dfrac{1}{2L_i}\| \Delta g_i^k\|^2 - \dfrac{1}{\lambda}\langle \Delta g_i^k, \Delta z_i^k \rangle , \label{eq:diffs}
\end{align}
 Meanwhile, we have from \eqref{eq:Deltaz1} that 
\begin{equation}
    \langle \Delta g_1^k, \Delta z_1^k \rangle  =  \gamma \| \Delta g_1^k \|^2 + \langle \Delta g_1^k,   \Delta x_1^k \rangle .
    \label{eq:inner1}
\end{equation} 
On the other hand, \eqref{eq:Deltaz2} yields 
\begin{equation}
      \langle \Delta g_2^k, \Delta z_2^k \rangle  = \gamma \|\Delta g_2^k\|^2 +  \alpha \langle \Delta g_2^k,    \Delta x_2^k \rangle - \alpha \langle \Delta g_2^k,   \Delta x_1^k \rangle .
      \label{eq:inner2}
\end{equation}
Combining \eqref{eq:diffs}, \eqref{eq:inner1} and \eqref{eq:inner2} gives the desired inequality. \qed 
\end{proof}

\begin{lemma}
\label{lemma:lower_alpha_epsilons}
    Let $\lambda\in (0,2)$ and let $\underline{\alpha} \coloneqq \frac{2\lambda - 3 + \sqrt{9 - 4\lambda}}{2}$. Then the following holds:
    \begin{enumerate}
        \item[(i)] $\underline{\alpha}\in \left( \frac{\lambda}{2}, 1\right)$.
        \item[(ii)] The interval $ \left( \dfrac{\alpha}{2\alpha-\lambda},\dfrac{2-\lambda}{1-\alpha}\right)$ is nonempty for any  $\alpha \in (\underline{\alpha},1)$.
        \item[(iii)]  For $\epsilon_1,\epsilon_2>0$ and $\alpha \in (\underline{\alpha},1)$, define the constants
    \begin{equation}
        \begin{array}{rcl}
        \bar{\gamma}_1 & \coloneqq & \dfrac{\lambda}{2L_2} - \dfrac{\alpha}{2\varepsilon_2} ,\\
        \bar{\gamma}_2 & \coloneqq & \dfrac{\alpha (2-\lambda - (1-\alpha)\epsilon_1)}{\alpha \epsilon_2 + 2(1-\alpha)L_1} ,\\
        \bar{\gamma}_3 & \coloneqq &\dfrac{(1-\alpha)(\epsilon_1 (2\alpha - \lambda)-\alpha)}{2\alpha L_2\epsilon_1},
        \end{array}
        \label{eq:upperbounds}
    \end{equation} 
    and the intervals
    \begin{equation}
             I_1 \coloneqq  \left( \frac{\alpha}{2\alpha-\lambda}, \frac{2-\lambda}{1-\alpha}\right) \quad \text{and} \quad  I_2  \coloneqq  \left( \frac{\alpha L_2}{\lambda},+\infty\right).
        \label{eq:intervals}
    \end{equation}
    If $\epsilon_j\in I_j$ for $j=1,2$, 
    then $\bar{\gamma}_i$ is strictly positive for $i=1,2,3$.
    \end{enumerate}
\end{lemma}
\begin{proof}
   {Part (i) is clear from the definition of $\underline{\alpha}$. To prove (ii), note first that for $\alpha\in (\underline{\alpha},1)$, we have $1-\alpha>0$ and $2\alpha - \lambda>0$ by part (i). Hence, the indicated interval in (ii) is nonempty provided that  $ (2-\lambda)(2\alpha - \lambda)>\alpha (1-\alpha)$. The latter condition holds true if and only if  $\alpha^2 + \alpha(3-2\lambda) - \lambda(2-\lambda)>0$, which is equivalent to having either $\alpha <\frac{2\lambda - 3 - \sqrt{9-4\lambda}}{2}$ or $\alpha > \frac{2\lambda - 3 + \sqrt{9-4\lambda}}{2} =  \underline{\alpha} $. Thus, the claim of part (ii) follows. Finally, part (iii) follows from straightforward calculations.}  \qed 
\end{proof}
With these lemmas in place, we are now ready to present the first main result of this paper. We establish the sufficient descent property of the \RtE, provided that the stepsize is chosen sufficiently small. In particular, we restrict the stepsize $\gamma$ in the interval 
\begin{equation}
    \Gamma \coloneqq \left(0, \min \{ \bar{\gamma}_0, \bar{\gamma}_1 \}\right] \cap \left(0, \min \left\{\bar{\gamma}_2,\bar{\gamma}_3 ,{\frac{1}{L_1+L_2}} \right\} \right),
    \label{eq:stepsizeinterval}
\end{equation}
where  $\bar{\gamma_0}\coloneqq \frac{\lambda}{2L_1}$, and$\bar{\gamma}_i$ is given by \eqref{eq:upperbounds}  for $i=1,2,3$, with $\epsilon_j$ taken from  $I_j$ given in \eqref{eq:intervals} for $j=1,2$. 


\begin{theorem}[\RtE~sufficient descent]\label{thm:Ryu_sufficient_descent}
Suppose that Assumption~\ref{a:blanket} holds. Let $\alpha \in (\underline{\alpha},1)$ where $\underline{\alpha} \coloneqq \frac{2\lambda - 3 + \sqrt{9 - 4\lambda}}{2}$, and $\lambda\in (0,2)$. Let  $\gamma \in \Gamma$, where $\Gamma$ is given in \eqref{eq:stepsizeinterval}. For the sequences $(z_i^k)_k$ for $i=1,2$ generated by \eqref{Ryu:iteration_relaxed}, there exists $M=M(\gamma)>0$ such that for all {$k\geq 0$}, \begin{equation*}
    \varphi_{\gamma}^{\texttt{Ryu}}(z_1^k,z_2^k)  \geq \varphi_{\gamma}^{\texttt{Ryu}}(z_1^{k+1},z_2^{k+1}) + M (\|z_1^{k+1} - z_1^k\|^2+\|z_2^{k+1} - z_2^k\|^2).
\end{equation*}
 
\end{theorem}

\begin{proof}
    From the definition of the relaxed Ryu envelope, we have 
    \begin{equation*}
    \begin{array}{rl}
         & \varphi_{\gamma}^{\texttt{Ryu}}(z_1^{k+1},z_2^{k+1}) \\
         & \leq f_3(x_3^k)  + \sum_{i=1}^2 \left(f_i(x_i^{k+1}) + \langle \nabla f_i(x_i^{k+1}), x_3^k - x_i^{k+1}\rangle +  \dfrac{1}{2\gamma_i} \|x_3^k-x_i^{k+1}\|^2 \right).
    \end{array}
    \end{equation*}
    Thus, together with Lemma \ref{lemma:Ryu:envelope-derivation}, we have 
    \begin{equation*}
        \begin{array}{rl}
           &  \varphi_{\gamma}^{\texttt{Ryu}}(z_1^k,z_2^k) - \varphi^{\texttt{Ryu}}_{\gamma}(z_1^{k+1},z_2^{k+1}) \\
           &  \geq \displaystyle\sum_{i=1}^2 \left(f_i(x_i^k) - f_i(x_i^{k+1}) + \langle \nabla f_i(x_i^k), x_3^k - x_i^k\rangle - \langle \nabla f_i(x_i^{k+1}), x_3^k - x_i^{k+1}\rangle \right) \\ 
           &   +  \displaystyle\sum_{i=1}^2 \dfrac{1}{2\gamma_i} \left( \|x_3^k-x_i^k\|^2 - \|x_3^k-x_i^{k+1}\|^2  \right)
        \end{array}
    \end{equation*} 
    Using Lemmas~\ref{lemma:diff_of_sqrd_norms} and~\ref{lemma:Ryu-lemma:L-smooth}, we obtain   
    \begin{equation*}
    \begin{array}{rl}
           &\varphi_{\gamma}^{\texttt{Ryu}}(z_1^k,z_2^k) - \varphi^{\texttt{Ryu}}_{\gamma}(z_1^{k+1},z_2^{k+1}) \\
           
           &  \geq  \left(\dfrac{1}{2L_1} -\dfrac{\gamma}{\lambda} \right)\norm{\Delta g_1^k}^2 + \left( \dfrac{1}{2L_2} - \dfrac{\gamma}{\lambda} \right) \norm{\Delta g_2^k}^2 \\
           & + \left( \dfrac{\alpha }{\gamma \lambda} - \dfrac{\alpha}{2\gamma}\right)\norm{\Delta x_1^k}^2 + \left(\dfrac{\alpha(1-\alpha) }{\gamma\lambda} -\dfrac{1-\alpha}{2\gamma}\right) \norm{\Delta x_2^k}^2 \\
           & + \dfrac{\alpha-1}{\lambda}\inner{\Delta x_1^k}{\Delta g_1^k}  + \dfrac{1-2\alpha}{\lambda}\inner{\Delta x_2^k}{\Delta g_2^k}  \\
           & - \dfrac{(1-\alpha)\alpha}{\gamma\lambda}\inner{\Delta x_2^k }{\Delta x_1^k} + \dfrac{\alpha}{\lambda}\inner{\Delta x_1^k}{\Delta g_2^k}.
        \end{array}
    \end{equation*} 
Since $\alpha<1$, then
    \begin{align*}
         \dfrac{\alpha-1}{\lambda}\inner{\Delta x_1^k}{\Delta g_1^k}  & \geq \dfrac{(\alpha-1)L_1}{\lambda}  \norm{\Delta x_1^k}^2, \\
          \dfrac{1-2\alpha}{\lambda}\inner{\Delta x_2^k}{\Delta g_2^k}  & = \dfrac{1-\alpha}{\lambda}\inner{\Delta x_2^k}{\Delta g_2^k} - \dfrac{\alpha}{\lambda}\inner{\Delta x_2^k}{\Delta g_2^k} \geq 0 - \dfrac{\alpha L_2}{\lambda}  \norm{\Delta x_2^k}^2 ,
    \end{align*}
where the first inequality holds by $L_1$-smoothness of $f_1$, while the second inequality holds by the convexity of $f_2$ and $L_2$-smoothness of $ f_2$. Moreover, by Young's inequality, we have 
    \begin{align*}
            - \dfrac{(1-\alpha)\alpha}{\gamma\lambda}\inner{\Delta x_2^k }{\Delta x_1^k} & \geq - \dfrac{(1-\alpha)\alpha}{\gamma\lambda} \left( \dfrac{\norm{\Delta x_2^k}^2}{2\epsilon_1} + \dfrac{\epsilon_1 \norm{\Delta x_1^k}^2}{2}\right) ,\\ 
          \dfrac{\alpha}{\lambda}\inner{\Delta x_1^k}{\Delta g_2^k}& \geq -\dfrac{\alpha}{\lambda}\norm{\Delta x_1^k} \norm{\Delta g_2^k} \geq -\frac{\alpha}{\lambda} \left( \dfrac{\norm{\Delta g_2^k}^2}{2\epsilon_2} + \dfrac{\epsilon_2\norm{\Delta x_1^k}^2}{2}\right),
    \end{align*}
where $\epsilon_1,\epsilon_2>0$ are arbitrary. With these, we obtain 
 \begin{equation}
    \begin{array}{rl}
           & \varphi_{\gamma}^{\texttt{Ryu}}(z_1^k,z_2^k) - \varphi^{\texttt{Ryu}}_{\gamma}(z_1^{k+1},z_2^{k+1})  \\
           & \geq  C_0\norm{\Delta g_1^k}^2 + C_1 \norm{\Delta g_2^k}^2+ C_2\norm{\Delta x_1^k}^2 +C_3\norm{\Delta x_2^k}^2 ,
        \end{array}
        \label{eq:suffdescent_almost}
    \end{equation} 
  where the constants $C_i$ for $i=0,1,2,3$ are given by\begin{align*}
        C_0 & = \dfrac{1}{2L_1} -\dfrac{\gamma}{\lambda}, \\
        C_1 & = \dfrac{1}{2L_2} - \dfrac{\gamma}{\lambda} - \dfrac{\alpha}{2\lambda\epsilon_2}, \\
        C_2 &= \dfrac{\alpha }{\gamma \lambda} - \dfrac{\alpha}{2\gamma} + \dfrac{(\alpha-1)L_1}{\lambda}  - \dfrac{(1-\alpha)\alpha\epsilon_1}{2\gamma\lambda} - \dfrac{\alpha\epsilon_2}{2\lambda}{\text{, and}} \\
        C_3 & = \dfrac{\alpha(1-\alpha) }{\gamma\lambda} -\dfrac{1-\alpha}{2\gamma}- \dfrac{\alpha L_2}{\lambda} -\dfrac{(1-\alpha)\alpha}{2\gamma\lambda\epsilon_1}.
    \end{align*}
Choosing $\epsilon_j\in I_j$ for $j=1,2$, it is not difficult to compute that $C_0,C_1\geq 0$ since $\gamma\leq \min\{ \bar{\gamma}_0,\bar{\gamma}_1 \}$. In addition,  $C_2,C_3> 0$ since $\gamma<\min\{ \bar{\gamma}_2,\bar{\gamma}_3 \}$. Hence, for the given stepsize $\gamma$, we obtain from \eqref{eq:suffdescent_almost} that
 \begin{equation}
    \begin{array}{rl}
           & \varphi_{\gamma}^{\texttt{Ryu}}(z_1^k,z_2^k) - \varphi^{\texttt{Ryu}}_{\gamma}(z_1^{k+1},z_2^{k+1}) \geq  C_2\norm{\Delta x_1^k}^2 +C_3\norm{\Delta x_2^k}^2 .
        \end{array}
        \label{eq:suffdescent_almost2}
    \end{equation} 
Meanwhile, from \eqref{eq:Deltaz1}, \eqref{eq:Deltaz2} and $L_i$-smoothness of $f_i$ for $i=1,2$, it follows \begin{equation*}
        \|\Delta z_1^k\|^2 \leq (1 + L_1\gamma)^2 \|\Delta x_1^k\|^2{\text{, and}}
    \end{equation*} \begin{equation*}
        \|\Delta z_2^k\|^2 \leq 2(\alpha + \gamma L_2)^2 \|\Delta x_2^k\|^2 + 2 \alpha^2 \|\Delta x_1^k\|^2.
    \end{equation*} Defining \begin{equation*}
        C_4 = \min\{C_2, C_3\}, \quad C_5 = \max\{ 2\alpha^2 + \big(1 + L_1\gamma\big)^2, 2(\alpha +  L_2 \gamma)^2\},
    \end{equation*} we  obtain from \eqref{eq:suffdescent_almost2} that \begin{equation*}
        \begin{array}{rcl}
             \varphi_{\gamma}^{\texttt{Ryu}}(z^k) - \varphi_{\gamma}^{\texttt{Ryu}}(z^{k+1}) & \geq & C_4( \|\Delta x_1^k\|^2 +  \|\Delta x_2^k\|^2)  \\
             & \geq & \dfrac{C_4}{C_5}( \|\Delta z_1^k\|^2 +  \|\Delta z_2^k\|^2).
        \end{array}
    \end{equation*} The proof concludes by setting $M = C_4/C_5 $. \qed
\end{proof}

Theorem~\ref{thm:Ryu_sufficient_descent} suggests that our modification of Ryu's three-operator splitting method behaves like a descent method for the suitably defined \RtE. Hence, one could expect this method to converge. In the next section, we formalize this idea.

\subsection{Convergence properties of the modified Ryu's algorithm}

The convergence analysis of the method in \eqref{Ryu:iteration_relaxed} relies on a particular relationship between the \RtE~and a  Lagrangian associated with a reformulation of problem \eqref{primal-problem}, which we proceed to define. By duplicating variables, problem \eqref{primal-problem} can be reformulated as follows \begin{equation*}
    \min_{x_1, x_2, x_3 \in \R^d} f_1(x_1) + f_2(x_2) + f_3(x_3) \quad \mbox{ s.t. } x_1 = x_3, \; x_2 = x_3.
\end{equation*} We define the following Lagrangian associated with this problem reformulation: \begin{equation*}
    \mathcal{L}_{\beta_1,\beta_2}(x_1,x_2, x_3, \mu_1, \mu_2) = f_3(x_3)+\sum_{i=1}^2 \left( f_i(x_i) + \langle \mu_i, x_i - x_3 \rangle  + \dfrac{\beta_i}{2} \|x_i - x_3\|^2\right),
\end{equation*} where, for $i=1,2$, $\mu_i \in \R^d$ is a Lagrangian multiplier associated with the constraint $x_i=x_3$. The next result states the aformentioned relationship between the \RtE~and the augmented Lagrangian.

\begin{lemma} \label{l:lag}
Suppose Assumption~\ref{a:blanket} holds. Let $\gamma \in (0, \frac{1}{L_1+L_2})$,  and $\lambda,\alpha>0$. Denoting \begin{equation*}
    \xi^k = \big(x_1^k,x_2^k, x_3^k, \gamma^{-1}(x_1^k - z_1^k), \gamma^{-1}(\alpha (x_2^k-x_1^k) - z_2^k)\big),
\end{equation*} then
    \begin{equation*}
    (\forall k \geq 1)\; \RE(z_1^k,z_2^k) = \mathcal{L}_{\frac{1}{\gamma_1},\frac{1}{\gamma_2}}(\xi^k).
\end{equation*}
\end{lemma}

\begin{proof}
    From \eqref{eq:z1^k}--\eqref{eq:z2^k}, $\nabla f_1(x_1^k) = \gamma^{-1}(z_1^k-x_1^k)$ and $\nabla f_2(x_2^k) = \gamma^{-1}(z_2^k-\alpha (x_2^k-x_1^k))$. Substituting these identities into \eqref{eq:RE} yields the result.
\end{proof}

We are now ready to state the second main result of this paper. We follow the approach taken in \cite{themelis2020douglas,Atenas25}.

\begin{theorem}[Subsequential convergence of nonconvex Ryu's three-operator splitting] \label{t:subseq-conv}
Suppose that Assumption~\ref{a:blanket} holds. Let $\alpha \in (\underline{\alpha},1)$ where $\underline{\alpha} \coloneqq \frac{2\lambda - 3 + \sqrt{9 - 4\lambda}}{2}$, and $\lambda\in (0,2)$. Let  $\gamma \in \Gamma$ such that {$\gamma \leq  \min \left\{ \frac{\alpha}{L_1},\frac{1-\alpha}{L_2} \right\} $}, where $\Gamma$ is given in \eqref{eq:stepsizeinterval}. Then, for any  sequence $((x_1^k,x_2^k,x_3^k), (z_1^k,z_2^k))_k$ generated by algorithm \eqref{Ryu:iteration_relaxed}, then
    \begin{itemize}
        \item[(i)] $x_i^{k+1} - x_i^k \to 0$ for $i=1,2,3$, and $z_i^{k+1} - z_i^k \to 0$ and $x_3^k - x_i^k \to 0$ for $i = 1,2$.
\end{itemize} 

 \noindent If, in addition, $((x_1^k,x_2^k,x_3^k), (z_1^k,z_2^k))_k$ is bounded, then
 \begin{itemize}
 \item[(ii)] For any cluster point $x^*$ of $(x_3^k)_k$,  $\varphi_{\gamma}(z^k) \to \varphi(x^*)$, and $f_1(x_1^k) + f_2(x_2^k)+ f_3(x_3^k) \to \varphi(x^*)$.
        \item[(iii)] All cluster points of the sequences $(x_i^k)_k$, for $i=1,2,3$, coincide and are critical points of problem \eqref{primal-problem}. 
    \end{itemize} 
\end{theorem}

\begin{proof}
    From the assumption, $\min \varphi > -\infty$, then   Proposition~\ref{p:fix-crit} implies $(\RE(z_1^k,z_2^k))_k$ is a bounded non-increasing sequence, thus convergent to some  $\varphi^{\star}_{\gamma} \in \R$. As a consequence, for $i = 1,2$, Theorem~\ref{thm:Ryu_sufficient_descent} yields $z_i^{k+1} - z_i^k \to 0$, and the $z_i^k$-update rule in \eqref{Ryu:iteration_relaxed} gives $x_3^k - x_i^k \to 0$. Since $f_1$ is convex, then $\prox_{\gamma f_1}$ is nonexpansive. In particular, \begin{equation*}
        \|x_1^{k+1} - x_1^k\| \leq \|z_1^{k+1} - z_1^k\|,
    \end{equation*} so that $x_1^{k+1} - x_1^k \to 0$. Likewise, $\prox_{\gamma f_2}$ is nonexpansive, yielding \begin{equation*}
        \|x_2^{k+1} - x_2^k\| \leq \dfrac{1}{\alpha}\|z_2^{k+1} - z_2^k\| + \|x_1^{k+1} - x_1^k\|,
    \end{equation*} and thus $x_2^{k+1} - x_2^k \to 0$. Moreover, as \begin{equation*}
        x_3^{k+1} - x_3^k = x_3^{k+1} - x_2^{k+1} + x_2^{k+1} - x_2^k + x_2^k - x_3^k,
    \end{equation*} then $x_3^{k+1} - x_3^k \to 0$ as well. This proves (i). Next,  observe that $x_3^k - x_1^k \to 0$ and $x_3^k - x_2^k \to 0$ imply that the sequences $(x_1^k)_k$, $(x_2^k)_k$ and $(x_3^k)_k$ have the same cluster points. Let $x^*$ be a cluster point of $(x_3^k)_k$ and $(z_1^*,z_2^*)$ be a cluster point of $(z_1^k,z_2^k)_k$, and let $x_i^{k_j} \to x^*$ for $i=1,2,3$, and $z_i^{k_j} \to z_i^*$ for $i=1,2$. Note that from continuity of the proximal operator, $x^* = \prox_{\gamma f_1}(z_1^*)$ and $x^*=\prox_{\frac{\gamma}{\alpha} f_2}(\frac{z_2^*}{\alpha} + x^*)$. Then \begin{equation*}
    \begin{array}{rll}
        \varphi(x^*) &\ds  \leq \liminf_j \varphi(x_3^{k_j}) & \quad \text{($\varphi$ is lsc)} \\
        &\ds   \leq \limsup_j \varphi(x_3^{k_j}) &  \\
        & \ds  \leq \limsup_j \RE(z_1^{k_j}, z_2^{k_j}) &  \quad \text{(Proposition~\ref{Ryu:sandwich}(ii))} \\
        &\ds  = \RE(z_1^*,z_2^*) & \quad \text{(Proposition~\ref{p:loc-lip})} \\
        & \leq \varphi(x^*).  & \quad \text{(Proposition~\ref{Ryu:sandwich}(i))}
    \end{array}
    \end{equation*} Hence, $\varphi^{\star}_{\gamma} = \lim_k \RE(z_1^k, z_2^k) = \lim_j \varphi(x_3^{k_j})  =\varphi(x^*)$. Furthermore, from Lemma~\ref{l:lag}, $\lim_k \mathcal{L}_{\frac{1}{\gamma_1},\frac{1}{\gamma_2}}(\xi^k) = \varphi(x^*)$. Since $(x_1^k)_k$, $(x_2^k)_k$, $(z_1^k)_k$ and $(z_2^k)_k$ are bounded, then part (i) implies $\lim_k \mathcal{L}_{\frac{1}{\gamma_1},\frac{1}{\gamma_2}}(\xi^k) = \lim_k f_1(x_1^k) + f_2(x_2^k)+ f_3(x_3^k) $, from where (ii) follows. Earlier in the proof we already showed the first part of item (iii). It remains to prove that $x^*$ is a critical point. A simple computation shows that for any $g_3 \in \partial f_3(x_3)$, \begin{multline*}
    \begin{pmatrix}
        \nabla f_1(x_1) + \mu_1 + \dfrac{1}{\gamma_1}(x_1 - x_3)\\ \nabla f_2(x_2) + \mu_2 + \dfrac{1}{\gamma_2}(x_2 - x_3)\\ g_3 - \left(\mu_1 + \mu_2 + \frac{1}{\gamma_1}(x_3 - x_1) + \frac{1}{\gamma_2}(x_3 - x_2)\right)
    \end{pmatrix}
           \\ \in \hat{\partial}_{(x_1,x_2, x_3)} \mathcal{L}_{\frac{1}{\gamma_1},\frac{1}{\gamma_2}}(x_1,x_2, x_3, \mu_1, \mu_2).
    \end{multline*} 
    Hence, taking $x_i = x_i^{k_j}$ for $i=1,2,3$, $$\mu_1^{k_j} = \gamma^{-1}(x_1^{k_j} - z_1^{k_j})\text{, and }\mu_2^{k_j} = \gamma^{-1}(\alpha (x_2^{k_j}-x_1^{k_j}) - z_2^{k_j}).$$ Then, in view of \eqref{eq:z1^k}, \eqref{eq:z2^k}, and \eqref{Ryu:OC3}, we get \begin{equation*}
    \begin{pmatrix}
         \dfrac{1}{\gamma_1}(x_1^{k_j} - x_3^{k_j})\\  \dfrac{1}{\gamma_2}(x_2^{k_j} - x_3^{k_j})\\ \dfrac{2\alpha}{\gamma}(x_1^{k_j}-x_2^{k_j})+\dfrac{2}{\gamma}(x_2^{k_j} - x_3^{k_j})
    \end{pmatrix}
            \in \hat{\partial}_{(x_1,x_2, x_3)} \mathcal{L}_{\frac{1}{\gamma_1},\frac{1}{\gamma_2}}(x_1^{k_j},x_2^{k_j}, x_3^{k_j}, \mu_1^{k_j}, \mu_2^{k_j}).
    \end{equation*} Taking the limit as $j \to \infty$, since  $x_i^{k_j} \to x^*$ for $i=1,2,3$, and $z_i^{k_j} \to z_i^*$ for $i=1,2$, item (i) then implies  \begin{equation*}
    \begin{pmatrix}
         0\\  0\\ 0
    \end{pmatrix}
            \in \partial_{(x_1,x_2, x_3)} \mathcal{L}_{\frac{1}{\gamma_1},\frac{1}{\gamma_2}}(x^{*},x^{*}, x^{*}, \mu_1^{*}, \mu_2^{*}),
    \end{equation*} where $\mu_1^* = \gamma^{-1}(x^* - z_1^*)$, and $\mu_2^* = -\gamma^{-1} z_2^*$. In turn, this inclusion is equivalent to the following conditions: \begin{align*}
        0 =& \nabla f_1(x^*) + \gamma^{-1}(x^* - z_1^*), \\
        0 =& \nabla f_2(x^*) -\gamma^{-1} z_2^*, \\
        0 \in & \partial f_3(x^*) -  \gamma^{-1}(x^* - z_1^* - z_2^*).
    \end{align*} Addind these equations yields $0\in  \nabla f_1(x^*) + \nabla f_2(x^*) + \partial f_3(x^*) = \partial \varphi (x^*)$, concluding the proof.
\end{proof}

\begin{remark}[Boundedness of the generated sequences] Using a similar argument to \cite[Theorem 3.4(iii)]{themelis2018forward}, we can show that if $\varphi$ has bounded level sets, then $\RE$ also has bounded level sets. Since for appropriately chosen stepsize, $(\RE(z_1^k, z_2^k))_k$ is a non-{increasing} sequence as shown in Theorem~\ref{thm:Ryu_sufficient_descent}, then for all $k\geq 1$, $(z_1^k, z_2^k) \in \{(z_1,z_2): \RE(z_1,z_2) \leq \RE(z_1^0,z_2^0)\}$, and thus $(z_1^k, z_2^k)_k$ is bounded. Moreover, boundedness of $(x_1^k,x_2^k)_k$ follows from the continuity properties of $\prox_{\gamma f_1}$ and $\prox_{\frac{\gamma}{\alpha} f_2}$, and boundedness of $(x_3^k)_k$ is a consequence of Theorem~\ref{t:subseq-conv}(i).
    
\end{remark}

\begin{remark}[Global convergence] In Theorem~\ref{t:subseq-conv}, we establish subsequential convergence of the relaxed variant of Ryu's three-operator splitting method in a specific nonconvex setting. A natural question is whether the method converges globally. This can be affirmatively addressed by adopting the now standard approach based on the Kurdyka--\L{}ojasiewicz (K\L{}) inequality~\cite{attouch2013convergence}, or alternatively, by leveraging the subdifferential-based error bound technique used in~\cite{Atenas25} within the unifying framework proposed in~\cite{atenas2023unified}. Under the same set of assumptions, one can also obtain local linear convergence rates.
\end{remark}

\section{{Numerical experiments}} \label{s:numerics}
In this section, we provide a numerical experiment to demonstrate the theoretical findings. The experiment we consider is nonnegative matrix completion/factorization
\cite{lee1999learning}, whose goal is to recover missing entries of a partially
observed matrix using nonnegative entries, and we impose the additional
structure requirement that the recovered matrix should be low-rank. 
\medskip

\noindent \textbf{Problem formulation.} Given a data matrix $M\in\mathbb{R}^{m\times n}$, let $\Omega$ be the set of observed
indices and $P_\Omega$ the associated projection, $(P_\Omega(X))_{ij}=X_{ij}$ if
$(i,j)\in\Omega$ and $0$ otherwise. We consider \eqref{primal-problem} with
\begin{align*}
f_1(X) &= \frac{\lambda_1}{2}\min_{Y\in\mathbb{R}^{m\times n}_+}\!\|X-Y\|_F^2,
\\
f_2(X) &= \tfrac12\|P_\Omega(X-M)\|_F^2,
\\
f_3(X) &= \lambda_2\sum_i \phi_{\mathrm{MCP}}\!\big(\sigma_i(X);\tau\big),
\end{align*}
where $\lambda_1,\lambda_2\ge 0$, $\|\cdot\|_F$ is the Frobenius norm,
$\{\sigma_i(X)\}$ are the singular values of $X$, and
$\phi_{\mathrm{MCP}}(\cdot;\tau)$ is the scalar minimax–concave penalty with $\tau>0$,
\[
\phi_{\mathrm{MCP}}(t;\tau)=
\begin{cases}
|t|-\dfrac{t^2}{2\tau}, & |t|\le \tau,\\[6pt]
\dfrac{\tau}{2}, & |t|>\tau.
\end{cases}
\]
This is the nonconvex {spectral MCP} penalty (a nonconvex alternative to the
nuclear norm; as $\tau\to\infty$, $f_3$ reduces to $\lambda_2\|X\|_*$). Note that $f_1$ is $\lambda_1$-smooth and $f_2$ is 1-smooth. The proximal operator of $f_1$ can be found in \cite{ProxRep},  the proximal operator of $f_2$ can be easily derived as 
\begin{equation*}
    \prox_{
\gamma f_2}(X) = X - \frac{\gamma}{1+\gamma}P_\Omega(X-M),
\end{equation*}
and the proximal operator of $f_3$ given in \cite[Corollary 2.1]{jin2016alternating}. 
\medskip 

\noindent \textbf{Problem data.} We follow \cite{toh2010accelerated} to generate $M\in\mathbb{R}^{m\times n}$ with rank
$r$ as the product of an $m\times r$ and an $r\times n$ matrix with i.i.d.\ standard
Gaussian entries, then select $s$ entries uniformly at random to form $\Omega$.
We set $m=n$ and consider $(n,s)\in\{(100,1000),(300,10000)\}$, fix $r=10$,
$\lambda_1=10$, $\lambda_2=5$, and use $\tau=100$ in $\phi_{\mathrm{MCP}}$.

\medskip
\noindent \textbf{Algorithm parameters.}
We implement the Ryu–splitting algorithm \eqref{Ryu:iteration_relaxed} with $\lambda=1$.
By Lemma~\ref{lemma:lower_alpha_epsilons}, $\underline{\alpha}=(-1+\sqrt{5})/2$, and we set
$\alpha=0.99\in(\underline{\alpha},1)$.
We choose $\varepsilon_1,\varepsilon_2>0$ to maximize the admissible step size
$\min\{\bar\gamma_1,\bar\gamma_2,\bar\gamma_3\}$: for fixed $\varepsilon_1$, we take
$\varepsilon_2=\varepsilon_2^{\star}(\varepsilon_1)$ as the positive solution of
$\bar\gamma_1=\bar\gamma_2$, then select $\varepsilon_1$ by solving
$\bar\gamma_1(\varepsilon_2^{\star}(\varepsilon_1))=\bar\gamma_3(\varepsilon_1)$ on
$\bigl(\alpha/(2\alpha-\lambda),(2-\lambda)/(1-\alpha)\bigr)$; if no interior solution
exists, we use the endpoint that yields the larger minimum.
With this, we set
\[
\gamma \;=\; \gamma_{\mathrm{Ryu}}
\;\coloneqq\; 0.99\,\min\!\left\{\bar\gamma_0,\,\bar\gamma_1,\,\bar\gamma_2,\,\bar\gamma_3,\,
\frac{1}{L_1+L_2}\right\},
\]
where $\bar\gamma_0=\frac{1}{2L_1}$.
For comparison, we also run Davis–Yin splitting~\cite{davis2017three} with the step size
$\gamma_{\mathrm{DYS}}$ prescribed in~\cite{ALT24} for the nonconvex model \eqref{primal-problem}.

We additionally consider an adaptive variant, denoted \emph{Ryu splitting}${+}$:
following \cite[Remark~4]{li2016douglas}, we initialize $\gamma=\gamma_0>0$ and, whenever
$\gamma>\gamma_{\mathrm{Ryu}}$ and either
$\min\{\|x_1^k-x_1^{k-1}\|,\ \|x_2^k-x_2^{k-1}\|\}>\tfrac{c_0}{k}$ or
$\min\{\|x_1^k\|,\ \|x_2^k\|\}>c_1$, we reduce $\gamma$ by a constant factor.
This preserves the guarantees of Theorem~\ref{t:subseq-conv}.
In our experiments, we set $\gamma_0=10\,\gamma_{\mathrm{DYS}}$, $c_0=10^{3}$,
$c_1=10^{10}$, and update $\gamma\leftarrow \max\{\gamma/2,\;0.99\,\gamma_{\mathrm{Ryu}}\}$. 

We terminate the algorithms when the fixed-point residual of the composite gradient–prox mapping falls
below a tolerance:
\[
\texttt{Res}(x)\;=\;\bigl\|\,x - \prox_{\hat{\gamma} f_3}\bigl(x-\hat{\gamma}\,(\nabla f_1(x)+\nabla f_2(x))\bigr)\,\bigr\|
\;<\;\varepsilon.
\]
We use $\hat{\gamma}=5\times 10^{-3}$, $\varepsilon=10^{-3}$, or stop after $20{,}000$ iterations.
\medskip

\noindent \textbf{Results.}
Table~\ref{tab:dys-ryu} reports averages over five random instances. With the theoretically prescribed step size from~\cite{ALT24}, the nonconvex DYS method reaches the target residual in remarkably fewer iterations and CPU time than Ryu splitting with the step size required by Theorem~\ref{t:subseq-conv}; in all cases, Ryu splitting hits the iteration cap (marked by $*$). Figure~\ref{fig:compare} corroborates this, showing much slower decay of both residual and objective under Ryu splitting. The key limiting factor is the conservativeness of the theoretical step-size bound. In contrast, the adaptive variant, Ryu splitting${+}$, employs larger effective step sizes and converges substantially faster than  DYS, suggesting that the current upper bound for Ryu’s step size may be sharpened in theory.

\begin{table}[t]
\centering
\caption{Numerical comparison of Davis--Yin splitting, Ryu splitting, and Ryu splitting+. Results are averaged over five random instances. ``Ryu splitting+'' denotes the adaptive version of Ryu splitting. An asterisk ($*$) indicates that the maximum iteration limit ($20{,}000$) was reached.}
\label{tab:dys-ryu}
\setlength{\tabcolsep}{6pt}
\begin{tabular}{lrrrrrrrrr}
\toprule
& \multicolumn{3}{c}{$n=100,\ s=1000$} &  \multicolumn{3}{c}{$n=300,\ s=10000$} \\
\cmidrule(lr){2-4}\cmidrule(lr){5-7}
Algorithm & iter & time & obj & iter & time & obj\\
\midrule
Davis--Yin splitting & 
    {6125} & {10.06} & {3534.1} & {14331} & {228.58} & {28137}   \\
Ryu splitting &
    {*20000} & {*34.05} & {*3534.1} & {*20000} & {*332.69} & {*28499}  \\
Ryu splitting+ &
    {624} & {1.17} & {3533.5} & {1555} & {26.37} & {28133} \\
\bottomrule
\end{tabular}
\end{table}

\begin{figure}[tb!]
	\centering
	\begin{tabular}{@{}cc@{}}
        \includegraphics[scale=0.4]{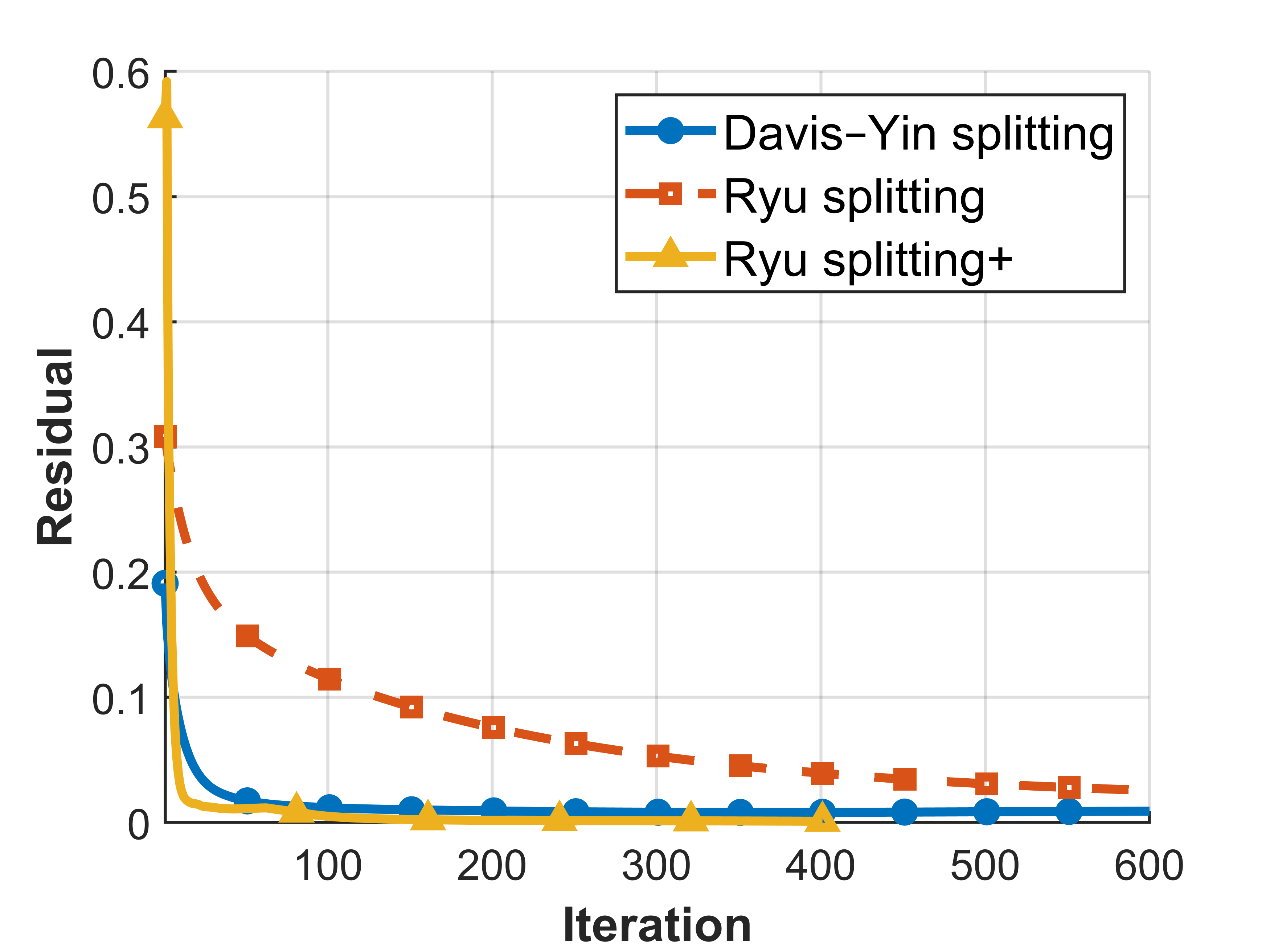}&
		\includegraphics[scale=0.4]{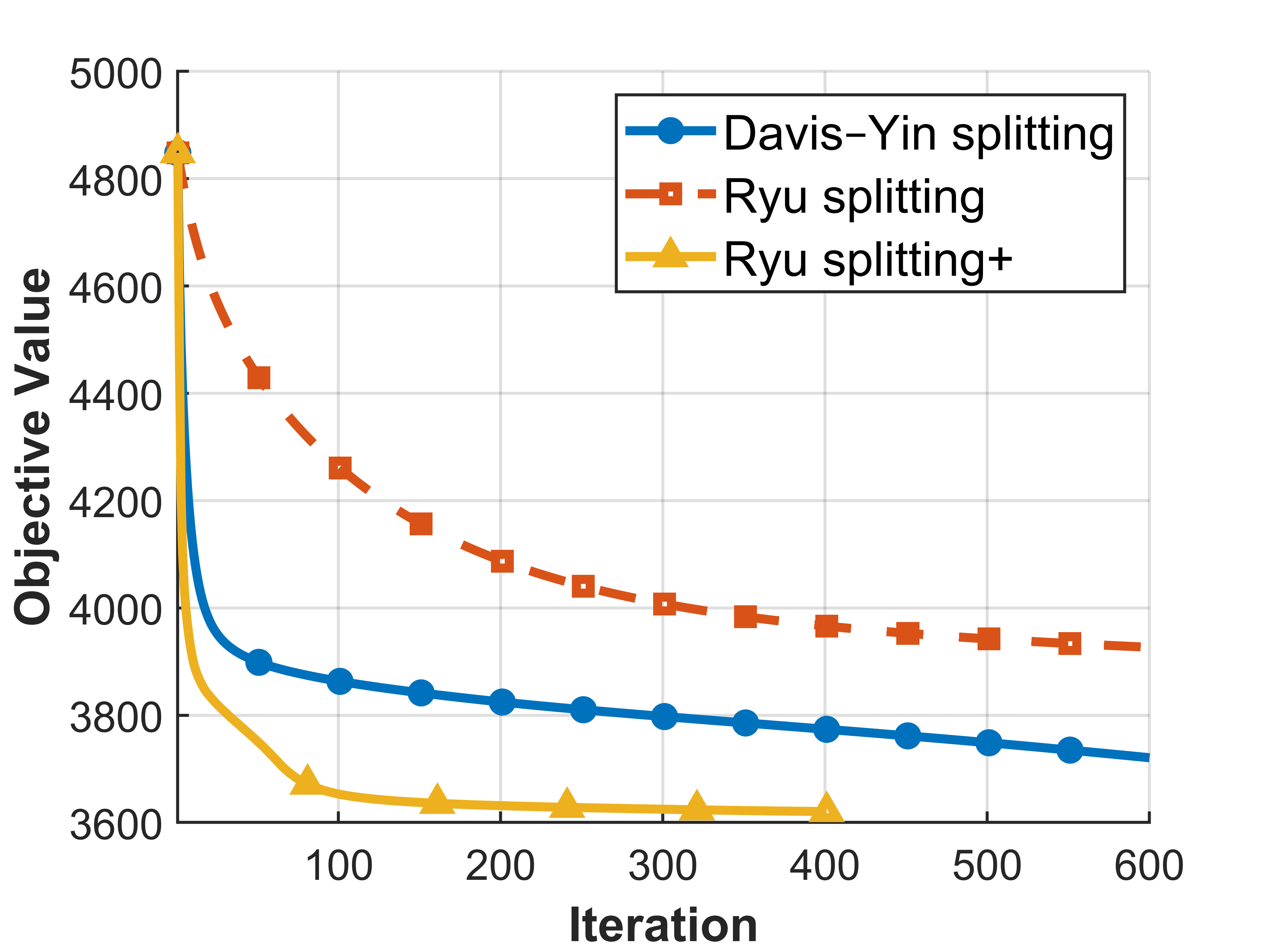}
	\end{tabular}
\caption{{Convergence on a random instance ($n=100$, $s=1000$). Left: residual vs.\ iteration; right: objective value vs.\ iteration.}}

	\label{fig:compare}
\end{figure}

\section{Conclusion} \label{s:conclusion}
By defining a Moreau-type envelope tailored to the algorithmic scheme in \eqref{Ryu:iteration_relaxed}, the core of the analysis relies on the sufficient decrease property shown in Theorem~\ref{thm:Ryu_sufficient_descent}. Observe that our analysis does not cover the limiting case $\alpha = 1$ (corresponding to the original method proposed by Ryu in \eqref{Ryu:iteration}), as it would make the stepsize interval in \eqref{eq:stepsizeinterval} empty. {This limitation occurs because, when $\alpha=1$, a key proximal term vanishes, thereby removing a structural component essential to our envelope-based descent analysis (see also Remark~\ref{rem:limitingcase}).} Nevertheless, when $\alpha$ is sufficiently close to 1, we can still guarantee global subsequential convergence for sufficiently small stepsizes. A full treatment of the limiting case $\alpha=1$ requires a more refined analysis, which is the subject of ongoing work 
by the authors.

\section*{Acknowledgments}
The authors thank the mathematical research institute MATRIX in Australia where part of this research was performed.
JHA's visit at MATRIX was supported in part by the MATRIX-Simons Travel Grant. The research of FA was supported in part by Australian Research Council grant DP230101749.

\bibliographystyle{abbrv}
\bibliography{bibfile}

\begin{thebibliography}{10}

\bibitem{ALT24}
J.~H. Alcantara, C.-p. Lee, and A.~Takeda.
\newblock A four-operator splitting algorithm for nonconvex and nonsmooth optimization.
\newblock {\em SIAM Journal on Optimization}, 35(3):1846--1872, 2025.

\bibitem{AT25}
J.~H. Alcantara and A.~Takeda.
\newblock Douglas-{R}achford algorithm for nonmonotone multioperator inclusion problems.
\newblock {\em arXiv preprint arXiv:2501.02752}, 2025.

\bibitem{Atenas25}
F.~Atenas.
\newblock Understanding the {D}ouglas--{R}achford splitting method through the lenses of {M}oreau-type envelopes.
\newblock {\em Computational Optimization and Applications}, 90:1--30, 2025.

\bibitem{atenas2023unified}
F.~Atenas, C.~Sagastiz{\'a}bal, P.~J. Silva, and M.~Solodov.
\newblock A unified analysis of descent sequences in weakly convex optimization, including convergence rates for bundle methods.
\newblock {\em SIAM Journal on Optimization}, 33(1):89--115, 2023.

\bibitem{attouch2013convergence}
H.~Attouch, J.~Bolte, and B.~F. Svaiter.
\newblock Convergence of descent methods for semi-algebraic and tame problems: proximal algorithms, forward--backward splitting, and regularized {G}auss--{S}eidel methods.
\newblock {\em Mathematical programming}, 137(1):91--129, 2013.

\bibitem{Beck17}
A.~Beck.
\newblock {\em First-Order Methods in Optimization}.
\newblock SIAM - Society for Industrial and Applied Mathematics, Philadelphia, PA, United States, 2017.

\bibitem{bian2021three}
F.~Bian and X.~Zhang.
\newblock A three-operator splitting algorithm for nonconvex sparsity regularization.
\newblock {\em SIAM Journal on Scientific Computing}, 43(4):A2809--A2839, 2021.

\bibitem{boyd2011distributed}
S.~Boyd, N.~Parikh, E.~Chu, B.~Peleato, J.~Eckstein, et~al.
\newblock Distributed optimization and statistical learning via the alternating direction method of multipliers.
\newblock {\em Foundations and Trends in Machine learning}, 3(1):1--122, 2011.

\bibitem{cai2010split}
J.-F. Cai, S.~Osher, and Z.~Shen.
\newblock Split {B}regman methods and frame based image restoration.
\newblock {\em Multiscale Modeling \& Simulation}, 8(2):337--369, 2010.

\bibitem{ProxRep}
G.~Chierchia, E.~Chouzenoux, P.~L. Combettes, and J.-C. Pesquet.
\newblock The proximity operator repository, 2016.

\bibitem{combettes2007douglas}
P.~L. Combettes and J.-C. Pesquet.
\newblock A {D}ouglas--{R}achford splitting approach to nonsmooth convex variational signal recovery.
\newblock {\em IEEE Journal of Selected Topics in Signal Processing}, 1(4):564--574, 2007.

\bibitem{combettes2011proximal}
P.~L. Combettes and J.-C. Pesquet.
\newblock {\em Proximal Splitting Methods in Signal Processing}, pages 185--212.
\newblock Springer New York, New York, NY, 2011.

\bibitem{davis2017three}
D.~Davis and W.~Yin.
\newblock A three-operator splitting scheme and its optimization applications.
\newblock {\em Set-valued and variational analysis}, 25:829--858, 2017.

\bibitem{DR}
J.~Douglas and H.~H. Rachford.
\newblock On the numerical solution of heat conduction problems in two and three space variables.
\newblock {\em Transactions of the American mathematical Society}, 82(2):421--439, 1956.

\bibitem{glowinski2017splitting}
R.~Glowinski, S.~Osher, and W.~Yin.
\newblock {\em Splitting Methods in Communication, Imaging, Science, and Engineering}.
\newblock Scientific Computation. Springer International Publishing, Cham, 2017.

\bibitem{jin2016alternating}
Z.-F. Jin, Z.~Wan, Y.~Jiao, and X.~Lu.
\newblock An alternating direction method with continuation for nonconvex low rank minimization.
\newblock {\em Journal of Scientific Computing}, 66(2):849--869, 2016.

\bibitem{lee1999learning}
D.~D. Lee and H.~S. Seung.
\newblock Learning the parts of objects by non-negative matrix factorization.
\newblock {\em Nature}, 401(6755):788--791, 1999.

\bibitem{li2016douglas}
G.~Li and T.~K. Pong.
\newblock {Douglas–Rachford} splitting for nonconvex optimization with application to nonconvex feasibility problems.
\newblock {\em Math. Program.}, 159(1-2):371--401, 2016.

\bibitem{LM}
P.-L. Lions and B.~Mercier.
\newblock Splitting algorithms for the sum of two nonlinear operators.
\newblock {\em SIAM Journal on Numerical Analysis}, 16(6):964--979, 1979.

\bibitem{MT23}
Y.~Malitsky and M.~K. Tam.
\newblock Resolvent splitting for sums of monotone operators with minimal lifting.
\newblock {\em Mathematical Programming}, 201(1):231--262, 2023.

\bibitem{Passty79}
G.~B. Passty.
\newblock Ergodic convergence to a zero of the sum of monotone operators in {H}ilbert space.
\newblock {\em Journal of Mathematical Analysis and Applications}, 72(2):383--390, 1979.

\bibitem{Rockafellar_Wets_2009}
R.~T. Rockafellar and R.~J.-B. Wets.
\newblock {\em Variational Analysis}, volume 317 of {\em Grundlehren der Mathematischen Wissenschaften}.
\newblock Springer Verlag Berlin, Berlin, 3rd printing edition, 2009.

\bibitem{ryu2020uniqueness}
E.~K. Ryu.
\newblock Uniqueness of {DRS} as the 2 operator resolvent-splitting and impossibility of 3 operator resolvent-splitting.
\newblock {\em Mathematical Programming}, 182:233--273, 2020.

\bibitem{themelis2020douglas}
A.~Themelis and P.~Patrinos.
\newblock Douglas--{R}achford splitting and {ADMM} for nonconvex optimization: Tight convergence results.
\newblock {\em SIAM Journal on Optimization}, 30(1):149--181, 2020.

\bibitem{themelis2018forward}
A.~Themelis, L.~Stella, and P.~Patrinos.
\newblock Forward-backward envelope for the sum of two nonconvex functions: Further properties and nonmonotone linesearch algorithms.
\newblock {\em SIAM Journal on Optimization}, 28(3):2274--2303, 2018.

\bibitem{toh2010accelerated}
K.-C. Toh and S.~Yun.
\newblock An accelerated proximal gradient algorithm for nuclear norm regularized linear least squares problems.
\newblock {\em Pacific Journal of optimization}, 6(615-640):15, 2010.

\end{thebibliography}

\end{document}